\documentclass[10pt]{amsart}

\usepackage{amssymb}
\usepackage{amsthm}
\usepackage{amsmath}

\usepackage[margin=1in]{geometry}

\usepackage[usenames]{color}
\definecolor{ANDREW}{RGB}{255,127,0}

\usepackage{tikz}
\usetikzlibrary{arrows,calc}

\usepackage[foot]{amsaddr}

\usepackage[normalem]{ulem}
\definecolor{darkred}{RGB}{229,0,51}
\definecolor{darkblue}{RGB}{0,0,153}
\definecolor{darkgreen}{RGB}{22,178,22}
\usepackage[colorlinks=true,urlcolor=darkblue,linkcolor=darkred,citecolor=darkgreen,pagebackref=true]{hyperref}

\theoremstyle{plain}
\newtheorem{proposition}{Proposition}[section]
\newtheorem{theorem}[proposition]{Theorem}
\newtheorem{lemma}[proposition]{Lemma}
\newtheorem{corollary}[proposition]{Corollary}
\theoremstyle{definition}

\newtheorem{definition}[proposition]{Definition}
\newtheorem{observation}[proposition]{Observation}
\theoremstyle{remark}
\newtheorem{remark}[proposition]{Remark}

\theoremstyle{plain}
\newtheorem{thmintro}{Theorem}

\newtheorem{corintro}[thmintro]{Corollary}

\DeclareMathOperator{\Aut}{Aut}

\DeclareMathOperator{\Vol}{Vol}

\DeclareMathOperator{\PGL}{PGL}

\DeclareMathOperator{\Ric}{Ric} 
\DeclareMathOperator{\Isom}{Isom} 
\DeclareMathOperator{\Div}{div}
\DeclareMathOperator{\Divw}{div^{\Wc}}

\DeclareMathOperator{\Bc}{\mathcal{B}}

\DeclareMathOperator{\Wc}{\mathcal{W}}

\DeclareMathOperator{\Jc}{\mathcal{J}}

\DeclareMathOperator{\Hb}{\mathbb{H}}

\DeclareMathOperator{\Nb}{\mathbb{N}}
\DeclareMathOperator{\Pb}{\mathbb{P}}
\DeclareMathOperator{\Rb}{\mathbb{R}}

\newcommand{\abs}[1]{\left|#1\right|}

\newcommand{\norm}[1]{\left\|#1\right\|}
\newcommand{\wt}[1]{\widetilde{#1}}
\newcommand{\wh}[1]{\widehat{#1}}
\newcommand{\ip}[1]{\left\langle #1\right\rangle}

\renewcommand{\leq}{\leqslant}
\renewcommand{\geq}{\geqslant}
\renewcommand{\epsilon}{\varepsilon}

\begin{document}

\title{Entropy rigidity of Hilbert and Riemannian metrics}
\author{Thomas Barthelm\'e}
\address{Department of Mathematics, Pennsylvania State University, University Park, State College, PA 16802}
\email{thomas.barthelme@queensu.ca}

\author{Ludovic Marquis}
\address{IRMAR, Universit\'e de Rennes, Rennes, France}
\email{ludovic.marquis@univ-rennes1.fr}

\author{Andrew Zimmer}
\address{Department of Mathematics, University of Chicago, Chicago, IL 60637.}
\email{aazimmer@uchicago.edu}

\date{\today}
\subjclass[2010]{}

\begin{abstract}
In this paper we provide two new characterizations of real hyperbolic $n$-space using the Poincar\'e exponent of a discrete group and the volume growth entropy. The first characterization is in the space of Riemannian metrics with Ricci curvature bounded below and generalizes a result of Ledrappier and Wang. The second is in the space of Hilbert metrics and generalizes a result of Crampon.
\end{abstract}

\maketitle

\section{Introduction}

Suppose $(X,d)$ is a proper metric space and $o \in X$ is some point. For any discrete group $\Gamma$ acting by isometries on $(X,d)$, we define the \emph{Poincar\'e}, or \emph{critical}, \emph{exponent} of $\Gamma$ as 
\[
 \delta_{\Gamma}(X,d) :=  \limsup_{r \rightarrow +\infty} \frac{1}{r} \log \# \{\gamma \in \Gamma \mid d(o, \gamma \cdot o ) \leq r \}.
\]
It is straightforward to show that this quantity does not depend on the choice of $o \in X$. If $X$ has a measure $\mu$ one can also define the \emph{volume growth entropy} as
\begin{equation*}
h_{vol}(X,d,\mu) := \limsup_{r \rightarrow +\infty} \frac{1}{r} \log \mu \left( B_r(o) \right)
\end{equation*}
where $B_r(o)$ is the open ball of radius $r$ about $o$. This quantity also does not depend on $o \in X$. If the measure $\mu$ is $\Isom(X,d)$-invariant, finite on bounded sets, and positive on open sets then a simple computation (see the proof of Lemma 4.5 in~\cite{Q2006}) shows
\begin{equation*}
\delta_\Gamma(X,d) \leq h_{vol}(X,d,\mu).
\end{equation*}
When additional assumptions are made, the Poincar\'e exponent and the volume growth entropy may coincide. For instance, if the action of $\Gamma$ on $(X,d)$ is cocompact, a simple argument shows that they are equal (again see the proof of Lemma 4.5 in~\cite{Q2006}).

These two invariants have a long and interesting history, as they are intimately related to the geometric and dynamical properties of the space $(X,d)$ (see for instance~\cite{M1979, FM1982}). Moreover, they are often linked to rigidity phenomenons (see for instance~\cite{BCG1995, BCG1996}).

In this paper we present two new characterizations of real hyperbolic $n$-space using the Poincar\'e exponent of a discrete group and the volume growth entropy. The first characterization (Theorem~\ref{thm:riem_finite_volume}) is in the space of Riemannian metrics with Ricci curvature bounded below and generalizes a result of Ledrappier and Wang~\cite{LW2010}. The second characterization (Theorem~\ref{thm:hilbert_finite_vol}) is in the space of Hilbert metrics and generalizes a result of Crampon~\cite{Cra2009}. This second result will follow from Theorem~\ref{thm:riem_finite_volume} and a recent result of Tholozan~\cite{Tho2015}.

\subsection{Riemannian metrics}\

Suppose $(X,g)$ is a complete, simply connected Riemannian $n$-manifold with $\Ric \geq -(n-1)$. Then the Bishop-Gromov volume comparison theorem implies that 
\begin{align*}
h_{vol}(X,g) \leq n-1
\end{align*}
(in the Riemannian case we always use the Riemannian volume form when considering the volume growth entropy). In particular, the volume growth entropy is maximized when $(X,g)$ is isometric to real hyperbolic $n$-space. There are many other examples which maximize volume growth entropy, but if $(X,g)$ has ``enough'' symmetry then it is reasonable to expect that $h_{vol}(X,g)=n-1$ if and only if $(X,g)$ is isometric to real hyperbolic $n$-space. This was recently proved by Ledrappier and Wang when $X$ covers a compact manifold: 

\begin{theorem}\cite{LW2010}
Let $(X,g)$ be a complete, simply connected Riemannian $n$-manifold and $\Gamma$ be a discrete group acting by isometries on $X$. Suppose that
\begin{enumerate}
 \item $\Ric \geq -(n-1)$;
 \item $\Gamma$ acts properly and freely on $X$ and $\Gamma \backslash X$ is compact;
 \item $h_{vol}(X,g)= n-1$.
\end{enumerate}
Then $X$ is isometric to the real hyperbolic space $\Hb^n$.
\end{theorem}

Our first new characterization of real hyperbolic space replaces compactness with finite volume, but with the cost of replacing $h_{vol}$ by $\delta_\Gamma$. 

\begin{thmintro} \label{thm:riem_finite_volume}
Let $(X,g)$ be a complete, simply connected Riemannian $n$-manifold and $\Gamma$ be a discrete group acting by isometries on $X$. Suppose that
\begin{enumerate}
 \item $\Ric \geq -(n-1)$;
 \item $X$ has bounded curvature;
 \item $\Gamma$ acts properly and freely on $X$ and $\Gamma \backslash X$ has finite volume;
 \item the Poincar\'e exponent satisfies $\delta_{\Gamma}(X,g)= n-1$.
\end{enumerate}
Then $X$ is isometric to the real hyperbolic space $\Hb^n$.
\end{thmintro}

\begin{remark} As in~\cite{LW2010}, it is possible to prove versions of Theorem~\ref{thm:riem_finite_volume} for K{\"a}hler or Quaternionic manifolds, but we will not pursue such matters here.
\end{remark}

Theorem~\ref{thm:riem_finite_volume} is a true generalization of Ledrappier and Wang's result: when $\Gamma \backslash X$ is assumed to be compact, $X$ has bounded curvature and the Poincar\'e exponent and the volume growth entropy coincide. Although our proof will follow the general outline of their argument, only assuming finite volume introduces a number of technical complications. Finally, the bounded curvature assumption is important  for our argument, but it may be possible to remove it. 

\subsection{Hilbert metrics}

Given a proper convex open set $\Omega \subset \Pb(\Rb^{n+1})$, we let $H_\Omega$ be the associated Hilbert metric. The Hilbert metric is a complete length metric on $\Omega$ which is invariant under the group of projective automorphisms of $\Omega$
\begin{equation*}
\Aut(\Omega) : = \{ \varphi \in \PGL_{n+1}(\Rb) : \varphi \Omega  = \Omega \}.
\end{equation*}
Moreover, if $\Omega$ is projectively equivalent to the ball $\Bc$, then $(\Omega, H_\Omega)$ is the Klein--Beltrami model of real hyperbolic $n$-space. 

Tholozan recently proved the following estimate for the volume growth entropy:

\begin{theorem}\cite{Tho2015} If $\Omega \subset \Pb(\Rb^{n+1})$ is a proper convex open set then 
\begin{equation*}
h_{vol}(\Omega, H_\Omega, \mu_B) \leq n-1
\end{equation*}
where $\mu_B$ is the Busemann--Hausdorff volume associated with $(\Omega, H_{\Omega})$ (or any bi-Lipschitz equivalent measure).
\end{theorem}

In particular the volume growth entropy is maximized when $\Omega$ is projectively equivalent to the unit ball. There are many other examples which maximize volume growth entropy, for instance, Berck, Bernig and Vernicos \cite{BBV2010} proved that, if $\partial \Omega$ is $C^{1,1}$ then 
\begin{equation*}
h_{vol}(\Omega, H_\Omega, \mu_B) = n-1.
\end{equation*}
However, once again, assuming that $\Omega$ has ``enough'' symmetry then one should expect that $h_{vol}(\Omega, H_\Omega, \mu_B)=n-1$ if and only if $\Omega$ is projectively equivalent to the unit ball. For instance, Crampon proved the following:

\begin{theorem} \cite{Cra2009}\label{thm:crampon}
Suppose $\Omega \subset \Pb(\Rb^{n+1})$ is a proper strictly convex open set and there exists a discrete group $\Gamma \leq \Aut(\Omega)$ that acts properly, freely, and cocompactly. Then $h_{vol}(\Omega, H_\Omega, \mu_B) \leq n-1$ with equality if and only if $\Omega$ is projectively isomorphic to $\Bc$ (and in particular $(\Omega, H_\Omega)$ is isometric to $\Hb^n$). 
\end{theorem}

\begin{remark} For the Hilbert metric, strict convexity of $\Omega$ is somewhat analogous to negative curvature. In particular, for a strictly convex set the Hilbert metric is uniquely geodesic, that is every pair of points are joined by a unique geodesic. Moreover, Benoist~\cite{B2004} proved that when $\Omega$ is strictly convex and has a compact quotient then the induced geodesic flow is Anosov and is $C^{1+\alpha}$. In his proof of Theorem~\ref{thm:crampon}, Crampon first shows that the topological entropy of this flow coincides with the volume growth entropy and then he uses techniques from hyperbolic dynamics to prove rigidity. For a general convex open set, the Hilbert metric may not be uniquely geodesic, but one can consider a natural ``geodesic line'' flow obtained by flowing along the geodesics that are lines segments in $\Pb(\Rb^d)$. However this flow is only $C^0$ and will have ``parallel'' flow lines. Thus Crampon's approach via smooth hyperbolic dynamics will not extend, 
at least directly, to the general case.
\end{remark}

Associated to every proper convex open set $\Omega \subset \Pb(\Rb^{n+1})$ is a Riemannian distance $B_\Omega$ on $\Omega$ called the \emph{Blaschke}, or \emph{affine}, distance (see, for instance, \cite{Lof2001,BH2013}). This Riemannian distance is $\Aut(\Omega)$-invariant and by a result of Calabi~\cite{C1972} has Ricci curvature bounded below by $-(n-1)$. In particular, if $d\Vol$ is the associated Riemannian volume form then the Bishop-Gromov volume comparison theorem implies that 
\begin{equation*}
h_{vol}(\Omega, B_\Omega, d\Vol) \leq n-1.
\end{equation*}

Benoist and Hulin \cite{BH2013} showed that the Hilbert distance and the Blaschke distance are bi-Lipschitz equivalent. Tholozan recently proved the following new relation:

\begin{theorem}\cite{Tho2015}\label{thm:tho}
If $\Omega \subset \Pb(\Rb^{n+1})$ is a proper convex open set, then 
\begin{equation*}
B_\Omega < H_\Omega +1.
\end{equation*}
In particular, 
\begin{equation*}
h_{vol}(\Omega, H_\Omega, \mu_B) \leq h_{vol}(\Omega, B_\Omega, d\Vol)
\end{equation*}
and if $\Gamma \leq \Aut(\Omega)$ is a discrete group then 
\begin{equation*}
\delta_\Gamma(\Omega, H_\Omega) \leq \delta_\Gamma(\Omega, B_\Omega).
\end{equation*}
\end{theorem}

Tholozan's result allows one to transfer from the Hilbert setting to the Riemannian setting where many more analytic tools are available. For instance, putting together Tholozan's result, the rigidity result of Ledrappier and Wang~\cite{LW2010} stated above, and some folklore properties of the Blaschke metric, one can remove the strictly convex hypothesis from Crampon's theorem (see Section~\ref{sec:Hilbert} for details):

\begin{theorem}\label{thm:hilbert_compact}
Suppose $\Omega \subset \Pb(\Rb^{n+1})$ is a proper convex open set and there exists a discrete group $\Gamma \leq \Aut(\Omega)$ which acts properly, freely, and cocompactly. Then $h_{vol}(\Omega, H_\Omega, \mu_B) \leq n-1$ with equality if and only if $\Omega$ is projectively isomorphic to $\Bc$ (and in particular $(\Omega, H_\Omega)$ is isometric to $\Hb^n$). 
\end{theorem}

Using our generalization of Ledrappier and Wang's result we get a second new characterization of real hyperbolic space.

\begin{thmintro}\label{thm:hilbert_finite_vol}
Suppose $\Omega \subset \Pb(\Rb^{n+1})$ is a proper convex open set and there exists a discrete group $\Gamma \leq \Aut(\Omega)$ which acts properly, freely, and with finite co-volume (with respect to $\mu_B$). Then $\delta_\Gamma(\Omega, H_\Omega) \leq n-1$ with equality if and only if $\Omega$ is projectively isomorphic to $\Bc$ (and in particular $(\Omega, H_\Omega)$ is isometric to $\Hb^n$). 
\end{thmintro}

When $\Gamma \backslash \Omega$ is non compact but has finite volume, it is unclear whether or not $h_{vol}(\Omega, H_\Omega, \mu_B)$ and $\delta_\Gamma(\Omega, H_\Omega)$ coincide (for Riemannian negatively curved metrics, there exists groups acting with finite co-volume for which the volume entropy and the critical exponent are distinct \cite{DPPS2009}). However, when $\Omega$ has $C^1$ boundary and is strictly convex then Crampon and Marquis~\cite[Th\'eor\`eme 9.2]{CM2014_geodesic_flow} proved that these two asymptotic invariants coincide. We will prove that in the finite volume quotient case having $C^1$ boundary and being strictly convex are equivalent and thus establish:

\begin{corintro}\label{cor:hilbert_finite_vol_2}
Suppose $\Omega \subset \Pb(\Rb^{n+1})$ is a proper convex open set which is either strictly convex or has $C^1$ boundary and such that there exists a discrete group $\Gamma \leq \Aut(\Omega)$ which acts properly, freely, and with finite co-volume (with respect to $\mu_B$). Then $h_{vol}(\Omega, H_\Omega, \mu_B) \leq n-1$ with equality if and only if $\Omega$ is projectively isomorphic to $\Bc$ (and in particular $(\Omega, H_\Omega)$ is isometric to $\Hb^n$). 
\end{corintro}

\begin{remark}
This result was announced for surfaces by Crampon in \cite{CraThese}, but his proof was not complete in the finite volume case since some of the dynamical results used are only fully proved in the compact case.
\end{remark}

\subsection*{Acknowledgments}
We would like to thank the referees for their careful reading of our article and their suggestions for improving it. The first author would like to thank Fran\c{c}ois Ledrappier and Nicolas Tholozan for helpful discussions. The second author was supported by the ANR Facettes and the ANR Finsler. The third author was partially supported by NSF grant 1400919.

\section{Entropy rigidity for Riemannian metrics}

This section is entirely devoted to the proof of Theorem \ref{thm:riem_finite_volume}. It will follow from Proposition \ref{prop:special_busemann} and Proposition \ref{prop:final_step} below.

\subsection{The Busemann boundary}

In this subsection we describe the Busemann compactification of a non-compact complete Riemannian manifold $(X,g)$. 

Fix a point $o \in X$. As in~\cite{L2010, LW2010}, we will normalize our Busemann functions such that $\xi(o)=0$. Now, for each $y \in X$, define the Busemann function based at $y$ to be 
\begin{equation*}
b_y(x) := d(x,y)-d(y,o).
\end{equation*}
As each $b_y$ is $1$-Lipschitz, the embedding $y \rightarrow b_y \in C(X)$ is relatively compact when $C(X)$ is equipped with the topology of uniform convergence on compact subsets. The \emph{Busemann compactification} $\wh{X}$ of $X$ is then defined to be the closure of $X$ in $C(X)$. The \emph{Busemann boundary} of $X$ is the set $\partial \wh{X} = \wh{X} \setminus X$. We begin by recalling some features of this compactification.

\begin{theorem}
\label{thm:buse_bd_basic}
Let $(X,g)$ be a non-compact complete simply connected Riemannian manifold. Then
\begin{enumerate}
\item $X$ is open in $\wh{X}$, hence the Busemann boundary $\partial \wh{X}$ is compact.
\item The action of $\text{Isom}(X)$ on $X$ extends to an action on $\wh{X}$ by homeomorphisms and for $\gamma \in \text{Isom}(X)$ and $\xi \in \partial \wh{X}$ the action is given by 
\begin{equation*}
(\gamma \cdot \xi)(x) = \xi(\gamma^{-1}x)-\xi(\gamma^{-1}o).
\end{equation*}
\end{enumerate}
\end{theorem}

The first result can be found in~\cite[Proposition 1]{LW2010}. The second assertion is straightforward to prove. 

\subsection{Patterson-Sullivan measures}

\begin{definition} 
Let $(X,g)$ be a non-compact complete simply connected Riemannian manifold and $\Gamma \leq \Isom(X,g)$ a discrete subgroup with $\delta_{\Gamma} < \infty$. A family of measures $\{ \nu_x : x \in X\}$ on $\partial \wh{X}$ is a (normalized) \emph{Patterson-Sullivan measure} if
\begin{enumerate}
\item $\nu_o(\partial \wh{X})=1$,
\item for any $x,y \in X$ the measures $\nu_x,\nu_y$ are in the same measure class and satisfy 
\begin{equation*}
\frac{d\nu_x}{d\nu_y}(\xi) = e^{-\delta_\Gamma(\xi(x)-\xi(y))},
\end{equation*}
\item for any $g \in \Gamma$, $\nu_{gx} = g_*\nu_x$.
\end{enumerate}
\end{definition}

Following the standard construction of Patterson-Sullivan measures via the Poincar\'e series (see for instance Section 2 of~\cite{LW2010}) we obtain:

\begin{proposition}
Let $(X,g)$ be a non-compact complete simply connected Riemannian manifold and $\Gamma \leq \Isom(X,g)$ a discrete subgroup with $\delta_\Gamma < \infty$. Then there exists a Patterson-Sullivan measure $\{ \nu_x : x \in X\}$ on $\partial \wh{X}$.
\end{proposition}

\subsection{An integral formula}

Now suppose $(X,g)$ is a non-compact complete simply connected Riemannian manifold and $\Gamma \leq \Isom(X,g)$ is a discrete subgroup with $\delta_\Gamma < \infty$. Moreover, assume that $\Gamma$ acts properly and freely on $X$ and the quotient manifold $M = \Gamma \backslash X$ has finite volume (with respect to the Riemannian volume form). 

Following \cite{L2010,LW2010}, we introduce the \emph{laminated space}
\begin{equation*}
X_M = \Gamma \backslash (X \times \partial \wh{X}) 
\end{equation*}
where $\Gamma$ acts diagonally on the product. The space $X_M$ is laminated by the images of $X \times \{ \xi \}$ under the projection. The leaves of this lamination inherit a smooth structure from $X$ and using this structure we can define a gradient $\nabla^{\Wc}$, a divergence $\text{div}^{\Wc}$, and a Laplacian $\Delta^{\Wc}$ in the leaf direction. A Patterson-Sullivan measure $\{\nu_x : x \in X\}$ yields a measure on the laminated space $X_M$ as follows: by definition $d\nu_x(\xi)=e^{-\delta_\Gamma\xi(x)}d\nu_o(\xi)$ for all $x \in X$. In particular if $dx$ is the Riemannian volume form on $X$, then the measure
\begin{equation*}
d\wt{m}(x,\xi)=e^{-\delta_\Gamma\xi(x)}dxd\nu_o(\xi)
\end{equation*}
is $\Gamma$-invariant and descends to a measure $\nu$ on $X_M$.

The argument at the end of Section 2 of~\cite{LW2010} can be used to show the following: 

\begin{theorem}\label{thm:integral_formula}
With the notation above, if $Y$ is a continuous vector field on $X_M$ which is $C^1$ along the leaves $X \times \{\xi\}$ such that $\norm{Y}_g$ and $\Div^{\Wc} Y$ are in $L^1(X_M, d\nu)$ then
\begin{equation*}
\int \Div^{\Wc} Y d\nu = \delta_\Gamma\int \left\langle Y,\nabla^{\Wc} \xi  \right\rangle d\nu.
\end{equation*}
\end{theorem}  

\begin{remark}
If $\xi \in \partial \wh{X}$, then $\norm{\nabla^{\Wc} \xi(x) } \leq 1$ for almost every $x \in X$. So we see that
\begin{equation*}
\int \abs{\left\langle Y,\nabla^{\Wc} \xi  \right\rangle} d\nu \leq \int \norm{Y}_g d\nu < \infty
\end{equation*}
and thus the right hand side of the equation in Theorem~\ref{thm:integral_formula} is well defined. 
\end{remark}

Now the function $x \rightarrow \nu_x(\partial \wh{X})$ is $\Gamma$-invariant so with a slight abuse of notation the measure $\nu$ has total mass
\begin{equation*}
\nu(X_M) = \int_M \nu_x(\partial \wh{X} ) dx.
\end{equation*}
Since $x \rightarrow \nu_x(\partial \wh{X})$ is continuous, if $M$ is compact then the measure $\nu$ is finite. For general finite volume quotients $\Gamma \backslash X$ it is not clear when $\nu$ will be a finite measure, but we can prove the following:

\begin{proposition}
With the notation above, if $(X,g)$ has $\Ric \geq -(n-1)$ and $\delta_{\Gamma} = n-1$ then $\nu(X_M) < \infty$.
\end{proposition}

\begin{proof}
Since $\Ric \geq -(n-1)$ the Laplacian comparison theorem implies for any $\xi \in \partial \wh{X}$ we have 
\begin{equation*}
\Delta e^{-(n-1) \xi} \geq 0
\end{equation*}
in the sense of distribution (see for instance~\cite[Proposition 4]{LW2010}). So in particular, since $\delta_{\Gamma} = n-1$, the function 
\begin{equation*}
f(x):= \nu_x(\partial \wh{X}) = \int_{\partial \wh{X}} e^{-(n-1)\xi(x)} d\nu_o(\xi)
\end{equation*}
is such that $\Delta f \geq 0$ in the sense of distributions. However, thanks to the invariance of the Patterson-Sullivan measure, $f$ is $\Gamma$-invariant and hence descends to a superharmonic function on $M = \Gamma \setminus X$. Since $M$ has finite volume and $f$ is a positive, superharmonic function, $f$ must be constant \cite[Proposition 0.2]{Ada1992}. Then,
\begin{equation*}
\nu(X_M) = \int_{M} \nu_x(\partial \wh{X}) dx =   \int_{M} \nu_o(\partial \wh{X}) dx =  \int_{M} dx = \Vol(M). \qedhere
\end{equation*}
\end{proof}

\subsection{A special Busemann function} This subsection is devoted to the proof of the following: 

\begin{proposition}\label{prop:special_busemann}
Suppose $(X,g)$ is a complete simply connected Riemannian manifold with $\Ric \geq -(n-1)$ and bounded sectional curvature. Assume $\Gamma \leq \Isom(X)$ is a discrete group that acts properly and freely on $X$ such that $M=\Gamma \backslash X$ has finite volume (with respect to the Riemannian volume form). If $\delta_{\Gamma} = n-1$ then there exists $\xi_0 \in \partial \wh{X}$ such that $\Delta \xi_0 \equiv n-1$.
\end{proposition}

For a general Riemannian manifold, the elements of $\partial \wh{X}$ are only Lipschitz. To overcome this lack of regularity we will consider smooth approximations obtained by convolution with the heat kernel. 

\begin{definition}\cite[Theorem 7.13]{G2009}
Suppose $(X,g)$ is a complete Riemannian manifold. The heat kernel $p_t(x,y) \in C^{\infty}(\Rb_{>0} \times X \times X)$ is the unique function satisfying:
\begin{enumerate}
\item $\frac{\partial}{\partial t} p_t = \Delta_x p_t = \Delta_y p_t$, 
\item $p_t(x,y) = p_t(y,x)$ for all $x,y  \in X$, 
\item $\lim_{t \searrow 0} p_t(x,y) = \delta_x(y)$ in the sense of distributions.
\end{enumerate}
\end{definition}

In the argument to follow it will also be helpful to use nicely behaved compactly supported functions:

\begin{lemma}\label{lem:nice_cut_off}
Suppose $M$ is a complete Riemannian manifold and $x_0 \in M$, then there exists $C>0$ such that for any $r>4$ there is a $C^\infty$ function $\varphi_r \colon M \rightarrow \Rb$ such that 
\begin{enumerate}
\item $0 \leq \varphi_r \leq 1$ on $M$, 
\item $\varphi_r \equiv 1$ on $B_r(x_0)$, 
\item $\varphi_r \equiv 0$ on $M \setminus B_{2r}(x_0)$, 
\item $\norm{\nabla \varphi_r} \leq C/r$ on $M$.
\end{enumerate}
\end{lemma}

\begin{proof}
Pick a smooth function $f \colon [0,\infty) \rightarrow \Rb$ such that $0 \leq f \leq 1$, $f \equiv 1$ on $[0,1]$, and $f \equiv 0$ on $[2,\infty)$. Let $C_1 = \max \{ \abs{f^\prime(t)}\}$. Next, let $g \colon [-1/3,4/3] \rightarrow [0,1]$ be a $C^\infty$ function with $g \equiv 0$ on $[-1/3,1/3]$ and $g \equiv 1$ on $[2/3,4/3]$. Let $C_2 = \max\{ \abs{g^\prime(t)}\}$. We claim that $C=2C_1C_2$ satisfies the conclusion of the lemma. 

Fix $r > 0$ and define the function $\phi: M \rightarrow \Rb$ by
\begin{equation*}
\phi(x) = f(d(x,x_0)/r).
\end{equation*}
Then $\phi$ is $C_1/r$-Lipschitz. Then, we can approximate $\phi$ by a $C^\infty$ function, $\theta \colon X \rightarrow \Rb$, so that $\abs{\phi - \theta} < 1/r$ and $\theta$ is $2C_1/r$-Lipschitz (see, for instance, \cite{AFR2007}). Finally, define 
\begin{equation*}
\varphi_r(x) := g(\theta(x)).
\end{equation*}
Then $0 \leq \varphi_r \leq 1$ on $N$ by construction. Moreover, if $x \in B_r(x_0)$, we have that $\phi(x) =1$ and so, $\theta(x) \in [1-1/r,1+1/r] \subset [2/3,4/3]$. Thus, $\varphi_r(x) = 1$. Similarly, if $x \in M \setminus B_{2r}(x_0)$ then $\varphi_r(x) = 0$. Finally, we see that $\varphi_r$ is $2C_1C_2/r$-Lipschitz. 
\end{proof}

For the rest of the subsection assume $(X,g)$ and $\Gamma \leq \Isom(X,g)$ satisfy the hypothesis of Proposition \ref{prop:special_busemann}. Let $p_t(x,y)$ be the heat kernel on $X$. By Theorem 4 in~\cite{CLY1981}: for any $t > 0$ there exists $C_p=C_p(t) \geq 1$, such that \begin{equation*}
p_t(x,y) \leq C_p e^{\frac{-d(x,y)^2}{C_p}} \text{ for all } x, y \in X.
\end{equation*}
On the space $X \times \partial \wh{X}$ define the function 
\begin{equation*}
F_t(x,\xi) := \int_X p_t(x,y) \xi(y) dy.
\end{equation*}
Because of the above estimate on $p_t(x,y)$, $F_t$ is well defined. In Appendix~\ref{sec:heat_kernel} we will use standard facts about the heat kernel to prove the following:

\begin{proposition}\label{prop:heat_kernel} With the notation above, 
\begin{enumerate}
\item For any $t > 0$ and $\xi \in \partial \wh{X}$, the function $x \rightarrow F_t(x,\xi)$ is $C^\infty$. 
\item For any $t > 0$, the functions $(x, \xi) \rightarrow \nabla_x F_t(x,\xi)$ and $(x, \xi) \rightarrow \Delta_x F_t(x,\xi)$ are continuous.
\item For any $t > 0$ and $\xi \in \partial \wh{X}$,
\begin{equation*}
\norm{\nabla_x F_t(x,\xi)} \leq e^{(n-1)t}.
\end{equation*}
\item For any $t > 0$ and $\xi \in \partial \wh{X}$,
\begin{equation*}
\Delta_x F_t(x,\xi) \leq n-1.
\end{equation*}
\end{enumerate}
\end{proposition}

Now, let $\wt{Y}_t(x,\xi) = \nabla_x F_t(x,\xi)$. Then $\wt{Y}_t$ descends to a continuous vector field $Y_t$ on $X_M$ which is $C^\infty$ along the leaves $X \times \{\xi\}$.

Next, let $\varphi_r \colon M \rightarrow \Rb$ be as in Lemma~\ref{lem:nice_cut_off} for some $x_0 \in M$. Then, define $\wt{f}_r \colon X \times \partial \wh{X} \rightarrow \Rb$ by $\wt{f}_r(x,\xi) = \varphi_r(\pi' (x))$, where $\pi':X \to M$ is the universal cover map. Since $\wt{f}_r$ is $\Gamma$-invariant, it descends to a continuous function $f_r \colon X_M \rightarrow \Rb$ which is $C^\infty$ along the leaves $X \times \{\xi\}$. 

Let $\wt{x}_0 \in X$ be a preimage of $x_0 \in M$. For $r > 0$, let $K_r \subset X_M$ be the image of $B_r(\wt{x}_0) \times \partial \wh{X}$ under the map 
\begin{equation*}
\pi\colon X \times \partial \wh{X} \rightarrow X_M.
\end{equation*}

\begin{lemma}
$K_r$ is compact, $f_r \equiv 1$ on $K_r$, and $f_r \equiv 0$ on $X_M \setminus K_{2r}$. 
\end{lemma}

\begin{proof}
Clearly, $K_r$ is compact by definition. Notice that $(x,\xi) \in \pi^{-1}(K_r)$ if and only if $x \in \cup_{\gamma \in \Gamma} B_r(\gamma \wt{x}_0)$. Thus, if $(x,\xi) \in \pi^{-1}(K_r)$ then $\wt{f}_r(x, \xi) \equiv 1$, and, if $(x,\xi) \notin \pi^{-1}(K_{2r})$ then $\wt{f}_r \equiv 0$. 
\end{proof}

\begin{lemma}
For any $r > 0$ and $t > 0$, 
\begin{equation*}
\norm{f_r Y_t} \in L^1(X_M, d\nu)
\end{equation*}
and 
\begin{equation*}
\Div^{\Wc} (f_r Y_t) \in L^1(X_M, d\nu).
\end{equation*}
\end{lemma}

\begin{proof}
Since $\norm{f_r Y_t}  \leq e^{(n-1)t}$, the first assertion is obvious. Now 
\begin{equation*}
\Div^{\Wc} (f_r Y_t )= f_r \Div^{\Wc} Y_t + \ip{\nabla^{\Wc} f_r,Y_t},
\end{equation*}
so 
\begin{equation*}
\int_{X_M} \abs{ \Div^{\Wc} f_r Y_t} d\nu \leq \int_{X_M} f_r \abs{ \Div^{\Wc} Y_t} d\nu + \frac{Ce^{(n-1)t}}{r} \nu(X_M).
\end{equation*}
However, the support of $f_r$ is compact in $X_M$ and the map $(x,\xi) \rightarrow \Delta_x F(x,\xi)$ is continuous. Thus, $ \abs{ \Div^{\Wc} Y_t} $ is bounded on the support of $f_r$. Hence,
\begin{equation*}
 \int_{X_M} f_r \abs{ \Div^{\Wc} Y_t} d\nu < +\infty. \qedhere
 \end{equation*}
\end{proof}

\begin{lemma}\label{lem:L1}
For any $t > 0$,
\begin{equation*}
\Div^{\Wc} Y_t \in L^1(X_M, d\nu),
\end{equation*}
and 
\begin{equation*}
\int_{X_M \setminus K_{2r}} \abs{ \Div^{\Wc} Y_t} d\nu \leq \left( 2n-2+\frac{C}{r} \right) e^{(n-1)t} \nu(X_M \setminus K_r).
\end{equation*}
\end{lemma}

\begin{proof}
For a real number $t$, let $t^+ = \max\{ 0, t\}$ and $t^- = \min\{ 0, t\}$. Then,
\begin{equation*}
\int_{X_M} \abs{\Div^{\Wc} (Y_t)} d \nu = \int_{X_M} \Div^{\Wc} (Y_t)^+ d \nu - \int_{X_M} \Div^{\Wc} (Y_t)^- d \nu
\end{equation*}
and, by Proposition~\ref{prop:heat_kernel},
\begin{equation*}
\int_{X_M} \Div^{\Wc} (Y_t)^+ d \nu \leq (n-1) \nu(X_M).
\end{equation*}
So, it is enough to bound the integral of $ \Div^{\Wc} (Y_t)^-$. 

By Theorem~\ref{thm:integral_formula},
\begin{equation*}
\int_{X_M} \Div^{\Wc} (f_r Y_t) d \nu =(n-1) \int_{X_M} \ip{ f_r Y_t, \nabla^{\Wc} \xi} d\nu.
\end{equation*}
So, by Proposition~\ref{prop:heat_kernel} (3) and the fact that $\norm{\nabla^{\Wc} \xi(x) } \leq 1$ for almost every $x \in X$,
\begin{equation*}
\abs{\int_{X_M} \Div^{\Wc} (f_r Y_t) d \nu} \leq (n-1) e^{(n-1)t} \nu(X_M). 
\end{equation*}
Now, 
\begin{equation*}
\Div^{\Wc} f_r Y_t = f_r \Div^{\Wc} Y_t + \ip{\nabla^{\Wc} f_r, Y_t} \quad \text{and} \quad \abs{ \ip{\nabla^{\Wc} f_r ,Y_t}} \leq \frac{C e^{(n-1)t}}{r},
\end{equation*}
so 
\begin{equation*}
\abs{\int_{X_M} f_r \Div^{\Wc} (Y_t) d \nu} \leq \left( \frac{C}{r} + n-1 \right)e^{(n-1)t}\nu(X_M). 
\end{equation*}
Then 
\begin{align*}
- \int_{X_M} f_r \Div^{\Wc}(Y_t)^- d\nu 
&= - \int_{X_M} f_r \Div^{\Wc}(Y_t) d\nu + \int_{X_M} f_r \Div^{\Wc}(Y_t)^+ d\nu \\
& \leq - \int_{X_M} f_r \Div^{\Wc}(Y_t) d\nu + (n-1) \nu(X_M).
\end{align*}
Which implies that
\begin{equation*}
- \int_{X_M} f_r \Div^{\Wc}(Y_t)^- d\nu \leq \left( \frac{C}{r} + 2n-2 \right)e^{(n-1)t}\nu(X_M). 
\end{equation*}
Finally $\lim_{r \rightarrow \infty} f_r = 1$ and so, by Fatou's Lemma,
\begin{equation*}
- \int_{X_M} \Div^{\Wc} (Y_t)^- d \nu \leq \liminf_{r \rightarrow \infty} - \int_{X_M} f_r \Div^{\Wc}(Y_t)^- d\nu \leq \left( 2n-2\right)e^{(n-1)t}\nu(X_M). 
\end{equation*}
By the remarks at the start of the proof we then have that $\Div^{\Wc} Y_t \in L^1(X_M, d\nu)$. 

To prove the second assertion, first observe that, for any $t \in \Rb$, $\abs{t} = -t+2t^+$. So, 
\begin{align*}
\int_{X_M \setminus K_{2r}}  \abs{\Div^{\Wc} Y_t}d\nu 
  &\leq \int_{X_M} (1-f_r) \abs{\Div^{\Wc} Y_t}d\nu \\
& =- \int_{X_M} (1-f_r) \Div^{\Wc} Y_td\nu+2\int_{X_M} (1-f_r) (\Div^{\Wc} Y_t)^+d\nu.
\end{align*}
Now,
\begin{equation*}
\int_{X_M} (1-f_r) (\Div^{\Wc} Y_t)^+d\nu \leq \int_{X_M} (1-f_r)(n-1) d\nu \leq (n-1) \nu(X_M\setminus K_r),
\end{equation*}
and, by Theorem~\ref{thm:integral_formula},
\begin{align*}
 - \int_{X_M} (1-f_r) \Div^{\Wc} Y_td\nu 
 &= - \int_{X_M}  \Div^{\Wc} \left( (1-f_r)Y_t\right)d\nu + \int_{X_M} \ip{ \nabla^{\Wc}(1-f_r), Y_t} d\nu  \\
& \leq (n-1) \abs{\int_{X_M} \ip{ (1-f_r) Y_t, \nabla^{\Wc} \xi} d\nu} + \frac{Ce^{(n-1)t}}{r} \nu(X_M \setminus K_r) \\
& \leq \left( n-1+\frac{C}{r} \right) e^{(n-1)t} \nu(X_M \setminus K_r).
\end{align*}
Combining the above inequalities establishes the second assertion of the lemma. 
\end{proof}

We finally have all the ingredients to prove Proposition \ref{prop:special_busemann}.

Using Lemma~\ref{lem:L1} we can apply Theorem \ref{thm:integral_formula} to $Y_t(x,\xi)$ and obtain: 
\begin{equation*}
0=\int_{X_M}\left( \Div^{\Wc} Y_t  -(n-1) \left\langle Y_t,\nabla^{\Wc} \xi  \right\rangle \right)d\nu.
\end{equation*}
Moreover, Lemma~\ref{lem:L1} implies that 
\begin{equation*}
\int_{X_M \setminus K_{2r}} \abs{ \Div^{\Wc} Y_t  -(n-1) \left\langle Y_t,\nabla^{\Wc} \xi  \right\rangle }d\nu 
\leq \left( 2n-2+\frac{C}{r} \right) e^{(n-1)t} \nu(X_M \setminus K_r) + (n-1) \nu(X_M \setminus K_{2r}).
\end{equation*}
So for any $\epsilon > 0$, there exists $r > 0$ such that
\begin{equation*}
\int_{X_M \setminus K_{2r}} \abs{ \Div^{\Wc} Y_t  -(n-1) \left\langle Y_t,\nabla^{\Wc} \xi  \right\rangle }d\nu 
\leq \epsilon
\end{equation*}
for all $0 \leq t \leq 1$.

We now choose a countable and locally finite open cover $\{U_i\}$ of $M$ such that each $U_i$ is small enough so that $\pi^{-1}(U_i)$ is a disjoint union of open sets all diffeomorphic to $U_i$. Let $\{\chi_i\}$ be a partition of unity subordinated to $\{U_i\}$. For each $U_i$, we choose one connected component of its lift that we denote by $\wt{U}_i$ and we write $\wt{\chi}_i$ for the lift of $\chi_i$ to $\wt U_i$.

\begin{observation}\label{obs:lifting} If $f \in L^1(X_M, d\nu)$ and $\wt{f}$ is the lift of $f$ to $X\times \partial \wh{X}$, then 
\begin{align*}
\int_{X_M} f d\nu = \sum_{i \in \Nb} \int_{x \in \wt{U}_i} \int_{\partial \wh{X}} \wt{\chi}_i(x) \wt{f}(x,\xi) e^{-(n-1)\xi(x)} d\nu_o dx.
\end{align*}
\end{observation} 

Next let
\begin{equation*}
\Jc := \{ j \in \Nb : U_j \cap B_{2r}(x_0) \neq \emptyset \}.
\end{equation*}
Because the cover $M = \bigcup U_i$ is locally finite, we see that $\Jc$ is a finite subset of $\Nb$. 

Notice that $\{ \chi_i \circ \pi\}$ is a partition of unity on $X_M$ and so
\begin{align*}
\int_{X_M}\left( \sum_{j \in \Jc} \chi_j \circ \pi \right) & \left( \Div^{\Wc} Y_t  -(n-1) \left\langle Y_t,\nabla^{\Wc} \xi  \right\rangle \right)d\nu 
= -\int_{X_M}\left( \sum_{j \notin \Jc} \chi_j \circ \pi \right) \left( \Div^{\Wc} Y_t  -(n-1) \left\langle Y_t,\nabla^{\Wc} \xi  \right\rangle \right)d\nu \\
& \geq - \int_{X_M \setminus K_{2r}} \abs{ \Div^{\Wc} Y_t  -(n-1) \left\langle Y_t,\nabla^{\Wc} \xi  \right\rangle}d\nu  \geq - \epsilon. 
\end{align*}
Moreover, by Observation~\ref{obs:lifting}
\begin{align*}
\int_{X_M} & \left( \sum_{j \in \Jc} \chi_j \circ \pi \right)  \left( \Div^{\Wc} Y_t  -(n-1) \left\langle Y_t,\nabla^{\Wc} \xi  \right\rangle \right)d\nu \\
& = \sum_{j \in \Jc} \int_{x \in \wt{U_j}}  \int_{\partial \wh{X}} \wt{\chi}_j(x) \left(\Div^{\Wc} \wt{Y}_t  -(n-1) \left\langle \wt{Y}_t,\nabla^{\Wc} \xi  \right\rangle \right)e^{-(n-1)\xi(x)} d\nu_o dx \\
&= \sum_{j \in \Jc} \int_{x \in \wt{U_j}}  \int_{\partial \wh{X}} \Divw\left( \wt{Y}_t  e^{-(n-1)\xi(x)}  \wt{\chi}_j \right)  - \langle  \wt{Y}_t , \nabla \wt{\chi}_j \rangle  e^{-(n-1)\xi(x)}  d\nu_o dx.
\end{align*}
Now because each $ \wt{\chi}_j$ is compactly supported in $\wt{U}_j$, Stokes Theorem implies that 
\begin{align*}
\int_{x\in \wt{U}_j} \Divw\left( \wt{Y}_t  e^{-(n-1)\xi(x)}  \wt{\chi}_j \right) dx=0
\end{align*}
and so by Fubini 
\begin{align*}
\epsilon \geq - \int_{X_M} & \left( \sum_{j \in \Jc} \chi_j \circ \pi \right)  \left( \Div^{\Wc} Y_t  -(n-1) \left\langle Y_t,\nabla^{\Wc} \xi  \right\rangle \right)d\nu \\
&= -\sum_{j \in \Jc}  \int_{\partial \wh{X}} \left( \int_{x \in \wt{U_j}}  \Divw\left( \wt{Y}_t  e^{-(n-1)\xi(x)}  \wt{\chi}_j \right)  - \langle  \wt{Y}_t , \nabla \wt{\chi}_j \rangle  e^{-(n-1)\xi(x)}  dx\right)d\nu_o \\
& = \sum_{j \in \Jc}  \int_{\partial \wh{X}} \left(\int_{x \in \wt{U_j}} \langle  \wt{Y}_t , \nabla \wt{\chi}_j \rangle  e^{-(n-1)\xi(x)}  dx\right)d\nu_o.
\end{align*}
Since the sum is finite, one can send $t \rightarrow 0$ to obtain
\begin{equation*}
 \sum_{j \in \Jc}  \int_{\xi \in \partial \wh{X}} \left(\int_{x\in \wt{U}_j}\langle  \nabla \xi , \nabla \wt{\chi}_j \rangle  e^{-(n-1)\xi(x)} dx \right)  d\nu_o \leq \epsilon.
\end{equation*}

By integration by parts, we have
\begin{align*}
 \sum_{j \in \Jc}  \int_{\xi \in \partial \wh{X}} \left(\int_{x\in \wt{U}_j}e^{-(n-1)\xi(x)} \Delta \wt{\chi}_j  dx \right)  d\nu_o
&= - \sum_{j \in \Jc} \int_{\xi \in \partial \wh{X}} \left(\int_{x\in \wt{U}_j}\langle  \nabla e^{-(n-1)\xi(x)} , \nabla \wt{\chi}_j \rangle   dx \right)  d\nu_o \\
&=  (n-1) \sum_{j \in \Jc}  \int_{\xi \in \partial \wh{X}} \left(\int_{x\in \wt{U}_j}\langle  \nabla \xi , \nabla \wt{\chi}_j \rangle  e^{-(n-1)\xi(x)} dx \right)  d\nu_o.
\end{align*}
\indent So, 
\begin{equation*}
 \sum_{j \in \Jc}  \int_{\xi \in \partial \wh{X}} \left(\int_{x\in \wt{U}_j}e^{-(n-1)\xi(x)} \Delta \wt{\chi}_j  dx \right)  d\nu_o  \leq \frac{\epsilon}{n-1}.
\end{equation*}

By \cite[Proposition 4]{LW2010} (that is still true in our context), $\Delta e^{-(n-1) \xi} \geq 0$ in the sense of distribution. Hence, for all $j \in \Jc$,
\[
\int_{\xi \in \partial \wh{X}} \int_{x\in \wt{U}_j}e^{-(n-1)\xi(x)} \Delta \wt{\chi}_j  dx  \geq 0.
\]
So, we conclude that for all $j \in\Jc$,
\[
\int_{\xi \in \partial \wh{X}} \int_{x\in \wt{U}_j}e^{-(n-1)\xi(x)} \Delta \wt{\chi}_j  dx  \leq  \frac{\epsilon}{n-1}.
\]
Since $\epsilon$ is arbitrarily small, we then deduce that for all $j \in \Nb$,
\[
 \int_{\xi \in \partial \wh{X}} \int_{x\in \wt{U}_j}e^{-(n-1)\xi(x)} \Delta \wt{\chi}_j  dx  =0.
\]
Then, exploiting the fact that $\Delta e^{-(n-1) \xi} \geq 0$ again, we see that for $\nu_0$-almost-every $\xi \in \partial \wh{X}$
\[
\int_{x\in \wt{U}_j}e^{-(n-1)\xi(x)} \Delta \wt{\chi}_j  dx  =0
\]
for every $j \in\Nb$.

\indent In the argument above, one can replace $\wt U_j$ by $g\cdot \wt U_j$ and $\wt \chi_j$ by $g\cdot \wt 
\chi_j$ for any $g \in \Gamma$. So for $\nu_0$-almost-every $\xi \in \partial \wh{X}$
\[
 \int_{x\in g \cdot \wt{U}_j}e^{-(n-1)\xi(x)} \Delta (g \cdot \wt{\chi}_j)  dx  =0
\]
for every $j \in\Nb$ and  every $g \in \Gamma$.
One can now conclude that, for  $\nu_0$-almost-every $\xi \in \partial \wh{X} $, $\Delta e^{-(n-1)\xi(x)} = 0$ in the sense of distribution in the same way as in \cite[p.472]{LW2010}, which concludes the proof of Proposition~\ref{prop:special_busemann}.

\subsection{Final steps}

\begin{proposition} \label{prop:final_step}
Suppose $(X,g)$ is a complete simply connected Riemannian manifold with $\Ric \geq -(n-1)$ and $\Gamma \leq \Isom(X)$ is a discrete group that acts properly and freely on $X$ such that $M=\Gamma \backslash X$ has finite volume (with respect to the Riemannian volume form). 

If there exists $\xi_0 \in \partial \wh{X}$ such that $\Delta \xi_0 \equiv n-1$ then $X$ is isometric to real hyperbolic $n$-space.
\end{proposition}

\begin{proof}
By the proof of Theorem 3.3 in~\cite{W2008} (also see the remark after Theorem 6 in~\cite{LW2010}) if there exists some $\xi_1 \in \partial \wh{X}$ such that $\Delta \xi_1 \equiv n-1$ and $\xi_1 \neq \xi_0$ then $X$ is isometric to real hyperbolic $n$-space. So, suppose for a contradiction, that we have 
\begin{equation*}
\{ \xi_0 \} = \{ \xi \in \partial \wh{X} : \Delta \xi \equiv n-1 \}.
\end{equation*} 
Since
\begin{equation*}
\Delta (\gamma \cdot \xi)(x) = (\Delta \xi)( \gamma^{-1} x)
\end{equation*}
we see that $\gamma \cdot \xi_0$ also has constant Laplacian equal to $n-1$. Thus $\gamma \cdot \xi_0 = \xi_0$ for all $\gamma \in \Gamma$.

Now if $\gamma \in \Gamma$ we see that
\begin{equation*}
\textrm{diff}(\gamma)_{\gamma^{-1} x} \nabla \xi_0(\gamma^{-1}x) = \nabla \Big( \xi_0(\gamma^{-1}x) \Big) = \nabla \Big(  \xi_0(\gamma^{-1}x)-\xi_0(\gamma^{-1}o) \Big) = \nabla (\gamma \cdot \xi_0)(x) = \nabla \xi_0(x).
\end{equation*}
Thus, $\nabla \xi_0(x)$ is a $\Gamma$-invariant vector field, and therefore descends to a vector field $V$ on $M$. 

Now, $\Div V = n-1$ since $\Div \nabla \xi_0 = \Delta \xi_0 \equiv n-1$, and moreover $\norm{V} \leq 1$. But, since $M$ has finite volume, there cannot exists a vector field $V$ with $\norm{V}, \Div V \in L^1(M)$ and $\Div V > 0$ (see for instance~\cite{K1981}). 
\end{proof}

Putting together Proposition \ref{prop:special_busemann} and Proposition \ref{prop:final_step} finishes the proof of Theorem \ref{thm:riem_finite_volume}.

\begin{remark}
Proposition \ref{prop:final_step} is actually not necessary for the proof of Theorem \ref{thm:riem_finite_volume}. Indeed, since we assume that $X$ has bounded curvature, we can replace Proposition \ref{prop:final_step} by \cite[Theorem 6]{LW2010}. However we included this result since it removes the need for the bounded curvature assumption from this step. In particular, we want to emphasize that the bounded curvature assumption is only used in order to get the heat kernel estimates needed for Proposition \ref{prop:special_busemann}.
\end{remark}

\section{Entropy rigidity for Hilbert metrics}\label{sec:Hilbert}

We begin by observing that the Blaschke metric has bounded sectional curvature. For the definition and some properties of the Blaschke metric, we refer to \cite{Lof2001,BH2013}.

\begin{lemma} \label{lem:bounded_curvature}
 Let $\Omega$ be a proper convex open set in $\Pb(\Rb^{n+1})$. There exists a universal constant $C_n$, depending only on the dimension such that the sectional curvature of the Blaschke metric on $\Omega$ is bounded above by $C_n$ and below by $-C_n$.
\end{lemma}

\begin{proof}
 Benz\'ecri \cite{Benz1960} proved that the action of $\mathrm{PGL}_{n+1}(\Rb)$ on the set of pointed proper convex open sets 
 \begin{align*}
 \mathcal{E}:= \{ (x,\Omega): \Omega \subset \Pb(\Rb^{n+1}) \text{ is a proper convex open set and } x \in \Omega\}
 \end{align*}
 is cocompact, so all we have to show is that the functions that, to an element $(\Omega, x) \in \mathcal{E}$ associates the maximum and minimum of the sectional curvature of the Blaschke metric at $x$, is $\mathrm{PGL}_{n+1}(\Rb)$-invariant and continuous. The invariance is clear from the definition of the Blaschke metric, and the continuity follows from Corollary 3.3 in \cite{BH2013}.
\end{proof}

We next prove Theorem~\ref{thm:hilbert_finite_vol} from the introduction:

\begin{theorem}
Suppose $\Omega \subset \Pb(\Rb^{n+1})$ is a proper convex open set and there exists a discrete group $\Gamma \leq \Aut(\Omega)$ which acts properly, freely, and with finite co-volume (with respect to $\mu_B$). Then $\delta_\Gamma(\Omega, H_\Omega) \leq n-1$ with equality if and only if $\Omega$ is projectively isomorphic to $\Bc$ (and in particular $(\Omega, H_\Omega)$ is isometric to $\Hb^n$). 
\end{theorem}

\begin{proof} Let $B_\Omega$ be the Blaschke metric on $\Omega$. Then 
\begin{enumerate}
\item $\Gamma$ acts by isometries on $(\Omega, B_\Omega)$ and the action is proper and free,
\item $B_\Omega$ has bounded sectional curvature by Lemma~\ref{lem:bounded_curvature}, 
\item  $B_\Omega$ has Ricci curvature bounded below by $-(n-1)$ by a result of Calabi~\cite{C1972},
\item by Theorem~\ref{thm:tho}, $\delta_\Gamma(\Omega, B_\Omega) = n-1$, 
\item by \cite[Proposition 2.6]{BH2013}, $\Gamma \backslash \Omega$ has finite volume with respect to the Riemannian volume form induced by $B_\Omega$. 
\end{enumerate}
Thus, the Blaschke metric satisfies all of the assumptions of Theorem \ref{thm:riem_finite_volume}, so $(\Omega, B_\Omega)$ is isometric to the real hyperbolic space. Hence, by definition of the Blaschke metric, $(\Omega, H_\Omega)$ is the Klein--Beltrami model of hyperbolic space (see \cite[Theorem 1]{Lof2001}). 
\end{proof}

Since $\delta_\Gamma(\Omega, H_\Omega) = h_{vol}(\Omega, H_\Omega, \mu_B)$ when $\Gamma$ acts co-compactly on $\Omega$ we immediately deduce Theorem~\ref{thm:hilbert_compact} from the introduction:

\begin{corollary}
Suppose $\Omega \subset \Pb(\Rb^{n+1})$ is a proper convex open set and there exists a discrete group $\Gamma \leq \Aut(\Omega)$ which acts properly, freely, and cocompactly. Then $h_{vol}(\Omega, H_\Omega, \mu_B) \leq n-1$ with equality if and only if $\Omega$ is projectively isomorphic to $\Bc$ (and in particular $(\Omega, H_\Omega)$ is isometric to $\Hb^n$). 
\end{corollary}

In order to prove Corollary \ref{cor:hilbert_finite_vol_2} from the introduction, we will need the following:

\begin{proposition}\label{p:vice_versa}
Suppose $\Omega \subset \Pb(\Rb^{n+1})$ is a proper convex open set and there exists a discrete group $\Gamma \leq \Aut(\Omega)$ which acts properly, freely, and with finite co-volume (with respect to $\mu_B$). Then $\Omega$ is strictly convex if and only if $\partial \Omega$ is $C^1$.
\end{proposition}

Before proving the proposition, we first deduce the corollary.

\begin{corollary}
Suppose $\Omega \subset \Pb(\Rb^{n+1})$ is a proper convex open set which is either strictly convex or has $C^1$ boundary and there exists a discrete group $\Gamma \leq \Aut(\Omega)$ which acts properly, freely, and with finite co-volume (with respect to $\mu_B$). Then $h_{vol}(\Omega, H_\Omega, \mu_B) \leq n-1$ with equality if and only if $\Omega$ is projectively isomorphic to $\Bc$ (and in particular $(\Omega, H_\Omega)$ is isometric to $\Hb^n$). 
\end{corollary}

\begin{proof}[Proof, assuming Proposition \ref{p:vice_versa}]
By Proposition~\ref{p:vice_versa}, $\partial \Omega$ is $C^1$ and $\Omega$ is strictly convex. Thus by~\cite[Th\'eor\`eme 9.2]{CM2014_geodesic_flow} 
\begin{equation*}
h_{vol}(\Omega, H_\Omega, \mu_B) = \delta_\Gamma(\Omega, H_\Omega).
\end{equation*}
So the corollary follows from Theorem~\ref{thm:hilbert_finite_vol}.
\end{proof}

\subsection{Proof of Proposition \ref{p:vice_versa}}

We begin by recalling some constructions and results. 

\subsubsection*{Duality} The dual of a proper convex open set of $\Pb(\Rb^{n+1})$ is the set 
\begin{align*}
\Omega^* = \{ \varphi  \in \Pb((\Rb^{n+1})^*) : \ker \varphi \cap \overline{\Omega} \neq \emptyset\}.
\end{align*}
It is straightforward to verify that $\Omega^*$ is a proper convex open subset of $\Pb((\Rb^{n+1})^*)$. 

Associated with a point $p \in \Pb(\Rb^{n+1})$ is the hyperplane $p^{**}$ in $\Pb((\Rb^{n+1})^*)$ consisting of all $\varphi \in\Pb((\Rb^{n+1})^*)$ with $\varphi(p) = 0$. For convex sets, we have the following connection between the boundaries of $\Omega$ and $\Omega^*$:

\begin{observation}
Suppose $\Omega \subset \Pb(\Rb^{n+1})$ is a proper convex open set. Then $p \in \partial \Omega$ if and only if $p^{**}$ is a supporting hyperplane of $\Omega^*$. Likewise, $\varphi \in \partial \Omega^*$ if and only if $\ker \varphi$  is a supporting hyperplane of $\Omega$.
\end{observation}

Now, if  $\Omega \subset \Pb(\Rb^{n+1})$ is a proper convex open set and $p \in \partial \Omega$, then $p$ is a $C^1$ point of $\partial \Omega$ if and only if there is a unique supporting hyperplane through $p$. So we have the following: 

\begin{observation}
Suppose $\Omega \subset \Pb(\Rb^{n+1})$ is a proper convex open set. Then $\Omega$ is strictly convex if and only if $\partial \Omega^*$ is $C^1$. Likewise, $\Omega^*$ is strictly convex if and only if $\partial \Omega$ is $C^1$.
\end{observation}

Finally, any element $\gamma \in \mathrm{Aut}(\Omega)$ acts on $\Omega^*$ via the action on the dual of $\Rb^{n+1}$, that is $\gamma^*(\varphi) = \varphi \circ\gamma^{-1}$, where $\varphi \in (\Rb^{n+1})^*$. We will denote by $\Gamma^* \leq \Aut(\Omega^*)$ the dual group of any subgroup $\Gamma$ of $\mathrm{Aut}(\Omega)$.

%
%
%
%
%

\subsubsection*{Margulis constant}

By \cite[Th\'eor\`eme 1]{CM_Margulis} or \cite[Theorem 0.1]{CLT2015}, in any dimension $n$, there exists a positive constant (called a \emph{Margulis constant}) $\varepsilon_n>0$, such that, for every proper convex open subset $\Omega$ of $\Pb(\Rb^{n+1})$, for every $x\in \Omega$, and for every $\Gamma$ discrete subgroup of $\mathrm{Aut}(\Omega)$, if $\Gamma_{\varepsilon_n}(x)$ is the group generated by the elements of $\Gamma$ that move $x$ at a distance less than $\varepsilon_n$ then $\Gamma_{\varepsilon_n}(x)$ is virtually nilpotent.

The \emph{thick part} of $\Omega$ is the closed subset of points $x \in \Omega$ such that $\Gamma_{\varepsilon_n}(x) = \{ 1\}$. The thick part of $\Gamma \backslash \Omega$ is the quotient of the thick part of $\Omega$ by $\Gamma$. The \emph{thin part} is the complement of the thick part.

If $x$ is inside the thick part then, by definition, the restriction to the ball of radius $\frac{\varepsilon_n}{2}$ of the projection $\Omega \rightarrow \Gamma \backslash \Omega$ is injective. The theorem of Benz\'ecri \cite{Benz1960} mentioned during the proof of Lemma \ref{lem:bounded_curvature} implies that the $\mu_B$-volume of the $H_{\Omega}$-ball of center $x$ and radius $\frac{\varepsilon_n}{2}$ is bounded from below by a constant independent of $x$ or $\Omega$. So, if the quotient $\Gamma \backslash \Omega$ has finite volume then the thick part of $\Gamma \backslash \Omega$ is compact, since it can contain only finitely many disjoint balls of radius $\frac{\varepsilon_n}{2}$. 

\subsubsection*{About the automorphisms of $\Omega$}
Let $\Gamma$ be a torsion-free finite type discrete subgroup of $\textrm{Aut}(\Omega)$. Suppose that $\Omega$ is strictly convex or with $C^1$-boundary, then each non-trivial element $\gamma \in \Gamma$ is either \emph{hyperbolic}, that is, $\gamma$ fixes exactly two points of $\overline{\Omega}$ which are on the boundary, or is \emph{parabolic}, that is, it fixes exactly one point of $\overline{\Omega}$ which is on the boundary (see \cite[Th\'eor\`eme 3.3]{CM_geo_fini},  or \cite[Proposition 2.8]{CLT2015}).

Suppose $G \leq \Gamma$ is a subgroup generated by two elements $\gamma,\delta$. Then,
\begin{enumerate}
\item either $G$ is virtually nilpotent and
\begin{enumerate}
\item either every element of $G$ is hyperbolic and has the same fixed points,
\item or every element of $G$ is parabolic and has the same fixed point; 
\end{enumerate}
\item or $G$ contains a free group and no point in $\overline{\Omega}$ is fixed by every element of $G$
\end{enumerate}
(see \cite[section 3.5]{CM_geo_fini},  or \cite[Proposition 4.13 and 4.14]{CLT2015}).

Let $\varepsilon >0$ be a Margulis constant. Let $\Lambda$ be a maximal parabolic subgroup of $\Gamma$, that is, a non-trivial stabilizer of a point $p\in \partial \Omega$ which contains one parabolic element (and hence, according to the above remarks, contains only parabolic elements). We define
$$
\Omega_{\varepsilon}(\Lambda) = \{ x \in \Omega \, | \,\exists \gamma \in \Lambda ,\, H_{\Omega}(x,\gamma x) \leqslant \varepsilon \}.
$$
This region is $p$-star-shaped, that is, for every $x \in \Omega_{\varepsilon}(\Lambda)$ the line segment joining $x$ to $p$ is contained in $ \Omega_{\varepsilon}(\Lambda)$. Moreover, if $\Lambda,\Lambda'$ are two distinct maximal parabolic subgroups of $\Gamma$ then $\Omega_{\varepsilon}(\Lambda) \cap \Omega_{\varepsilon}(\Lambda') =\varnothing$ (see \cite[Lemme 6.2]{CM_geo_fini}).

\begin{proof}[Proof of Proposition~\ref{p:vice_versa}]
By \cite[Corollary 6.7]{CLT2015}, the quotient $\Gamma \backslash \Omega$ has finite volume if and only if the dual quotient $\Gamma^* \backslash \Omega^*$ also has finite volume. Hence, we only have to show that if $\Omega$ is strictly convex and $\Gamma \backslash \Omega$ has finite volume then $\partial \Omega$ is of class $C^1$.

Suppose that $\Omega$ is strictly convex. We want to use \cite[Theorem 0.15]{CLT2015} to conclude that $\partial \Omega$ is of class $C^1$. In order to apply that theorem, we need to prove that $\Gamma \backslash \Omega$ is topologically tame and that the holonomy of each boundary component is parabolic.


Fix $\varepsilon >0$ a Margulis constant. Since $\Gamma \backslash \Omega$ has finite volume, the thick part of $\Gamma \backslash \Omega$ is compact (by the above remarks). Since $\Omega$ is strictly convex, a connected component $\mathcal{H}$ of the thin part is of one of two types. Either it is a lift of a Margulis tube, that is, a lift of a tubular neighborhood of a closed geodesic of length less than $\varepsilon$. Or it is preserved by a maximal parabolic subgroup $\Lambda$ of $\Gamma$ and $\mathcal{H}=\Omega_{\varepsilon}(\Lambda)$ (see \cite[Lemma 8.2]{CLT2015}).

Since there are only a finite number of geodesics of length less than $\varepsilon$, the thick part of $\Gamma \backslash \Omega$ together with all the Margulis tubes is still compact. Hence, if $\mathcal{H}$ is a connected component of the thin part and not a lift of a Margulis tube, the action of $\Lambda$ on $\partial \mathcal{H}$ is cocompact. Since, $\mathcal{H}$ is $p$-star shaped, the quotient $\Gamma \backslash \Omega$ is topologically tame and the holonomy of each boundary component is parabolic. We can thus apply \cite[Theorem 0.15]{CLT2015} to conclude that $\partial \Omega$ is of class $C^1$.
\end{proof}

\appendix

\section{Proof of Proposition~\ref{prop:heat_kernel}}\label{sec:heat_kernel}

For the rest of the section suppose that $(X,g)$ is a complete non-compact simply connected Riemannian manifold with $\Ric \geq -(n-1)$ and bounded sectional curvature. 

For a function $f \colon X \rightarrow \Rb$ define the function $P_t(f) \colon X \rightarrow \Rb$ by 
\begin{equation*}
P_t(f) (x) = \int_X p_t(x,y) f(y) dy. 
\end{equation*}

We will need some estimates on the heat kernel: 

\begin{lemma}\cite[Theorem 4 and Theorem 6]{CLY1981}\label{lem:hk_est}
With the notation above, for any $T > 0$, there exists $C > 0$ such that 
\begin{equation*}
p_t(x,y) \leq Ct^{-\frac{n}{2}}\exp \left( \frac{-d(x,y)^2}{Ct} \right)
\end{equation*}
and 
\begin{equation*}
\norm{\nabla_x p_t(x,y)}  \leq Ct^{-\frac{n+1}{2}}\exp \left( \frac{-d(x,y)^2}{Ct} \right)
\end{equation*}
for all $t \in (0,T]$ and $x,y \in X$.
\end{lemma}

\begin{proof}[Proof of Proposition~\ref{prop:heat_kernel}]
Recall that $F_t(x,\xi) = P_t(\xi)(x)$. We claim that for any $\xi \in \partial \wh{X}$ 
\begin{equation*}
(\partial_t - \Delta_x)F_t(x,\xi) = 0
\end{equation*}
in the sense of distributions. Once this is established part (1) and part (2) follow from standard regularity results (see for instance~\cite[Theorem 7.4]{G2009}). 

Let $\phi \in C_c^\infty(X \times \Rb_+)$. By Lemma~\ref{lem:hk_est}, $\Delta_x( \phi(x,t)) p_t(x,y) \xi(y)$ and $\partial_t (\phi(x,t)) p_t(x,y) \xi(y)$ are in  $L^1(X \times X \times \Rb_+, dxdydt)$. Then, using Fubini and the fact that $\partial_t p_t(x,y) = \Delta_x p_t(x,y)$, we obtain
\begin{align*}
\int_{X\times \Rb_+}&  \Delta_x \phi(x,t) P_t(\xi)(x) dxdt = \int_X \left( \int_{X\times \Rb_+} \Delta_x \phi(x,t) p_t(x,y) dxdt\right) \xi(y) dy \\
& = \int_X \left( \int_{X\times \Rb_+} \phi(x,t) \Delta_x  p_t(x,y) dxdt\right) \xi(y) dy = \int_X \left( \int_{X\times \Rb_+} \phi(x,t) \partial_t  p_t(x,y) dxdt\right) \xi(y) dy \\
&= -\int_X \left( \int_{X\times \Rb_+} \partial_t \phi(x,t) p_t(x,y) dxdt\right) \xi(y) dy = -\int_{X\times \Rb_+} \partial_t \phi(x,t) P_t(\xi)(x) dxdt.
\end{align*}
\indent Thus 
\begin{equation*}
(\partial_t - \Delta_x)F_t(x,\xi) = 0
\end{equation*}
in the sense of distributions. So part (1) and (2) are established. 

Now, by~\cite{BE1984}, since $\Ric \geq -(n-1)$, if $f \in C_c^\infty(X)$ then 
\begin{equation*}
\norm{\nabla P_t(f)}_{\infty} \leq e^{(n-1)t} \norm{\nabla f}_{\infty}.
\end{equation*}
Moreover, for any $\xi \in \partial \wh{X}$, there exists a sequence $f_m \in C_c^\infty(X)$ such that $f_m$ converges to $\xi$ locally uniformly and 
$\norm{\nabla f_m}_{\infty} \rightarrow 1$ (see, for instance, \cite{AFR2007}). Hence, each $P_t(f_m)$ is $e^{(n-1)t} \norm{\nabla f_m}_{\infty}$-Lipschitz. Moreover, by Lemma~\ref{lem:hk_est} and the dominated convergence theorem, $P_t(f_m)(x) \rightarrow P_t(\xi)(x)$ for all $x \in X$. Thus, $P_t(\xi)$ is $e^{(n-1)t}$-Lipschitz and 
\begin{equation*}
\norm{\nabla_x F_t(x,\xi)} \leq e^{(n-1)t}.
\end{equation*}

Now fix some non-negative $\phi \in C_c^\infty(X)$. Then we have
\begin{multline*}
\int_{X}  \Delta_x \phi(x) P_t(\xi)(x) dx  = \int_X \left( \int_{X} \phi(x) \Delta_x  p_t(x,y) dx\right) \xi(y) dy \\
 =  \int_X \left( \int_{X} \phi(x) \Delta_y  p_t(x,y) dx\right) \xi(y) dy = \int_X \Delta_y P_t(\phi)(y) \xi(y) dy.
\end{multline*}
For $r > 0$, let $\varphi_r : X \rightarrow \Rb$ be as in Lemma~\ref{lem:nice_cut_off}. Then 
\begin{multline*}
\int_X \Delta_y  P_t(\phi)(y) \xi(y) dy = \int_X \Delta_y \Big( \varphi_r(y) P_t(\phi)(y)\Big) \xi(y) dy + \int_X \Delta_y \Big( (1-\varphi_r)(y) P_t(\phi)(y)\Big) \xi(y) dy \\
 \leq (n-1) \int_X \varphi_r(y) P_t(\phi)(y) dy + \int_X \Delta_y \Big( (1-\varphi_r)(y) P_t(\phi)(y)\Big) \xi(y) dy.
\end{multline*}
Using the dominated convergence theorem once again, we have
\begin{equation*}
\lim_{r \rightarrow \infty}  \int_X \varphi_r(y) P_t(\phi)(y) dy = \int_X P_t(\phi)(y) dy = \int_X \phi(x) dx.
\end{equation*}

\indent Moreover, since integration by parts holds for Lipschitz functions,
\begin{multline*}
\abs{ \int_X \Delta_y \Big( (1-\varphi_r)(y) P_t(\phi)(y)\Big) \xi(y) dy  }=  \abs{\int_X\ip{ \nabla_y \Big( (1-\varphi_r)(y) P_t(\phi)(y)\Big), \nabla_y \xi(y)} dy} \\
  \leq \frac{C}{r} \int_X P_t(\phi)(y) dy + \int_{ X \setminus B_r(o)} \norm{\nabla P_t(\phi)(y)} dy.
\end{multline*}
Now,
\begin{equation*}
\nabla P_t(\phi)(y) = \int_X \nabla_y p_t(x,y) \phi(x) dx,
\end{equation*}
and so, by Lemma~\ref{lem:hk_est},
\begin{equation*}
\norm{\nabla P_t(\phi)(y)} \in L^1(X, dy).
\end{equation*}
Thus,
\begin{equation*}
\lim_{r \rightarrow \infty} \int_{ X \setminus B_r(o)} \norm{\nabla P_t(\phi)(y)} dy = 0.
\end{equation*}
Which implies that
\begin{align*}
\lim_{r \rightarrow \infty}  \int_X & \Delta_y \Big( (1-\varphi_r)(y) P_t(\phi)(y)\Big) \xi(y) dy = 0,
\end{align*}
and thus 
\begin{equation*}
\int_{X}  \Delta_x \phi(x) P_t(\xi)(x) dx \leq (n-1) \int_X \phi(x) dx.
\end{equation*}
Since $\xi \in \partial \wh{X}$ was arbitrary and $\phi \in C_c^\infty(X)$ is an arbitrary non-negative function, we see that 
\begin{equation*}
\Delta_x F_t(x,\xi) \leq n-1. \qedhere
\end{equation*}
\end{proof}

\bibliographystyle{alpha}
\bibliography{geom}

\end{document}